\documentclass[12pt,onepage]{amsart}
\usepackage[usenames,dvipsnames]{xcolor}
\usepackage[colorlinks=true,linkcolor=Blue,citecolor=Green]{hyperref}
\usepackage{amsmath,amsthm,amsfonts,amssymb}
\usepackage[T1]{fontenc}
\usepackage[utf8]{inputenc}
\usepackage[capitalise,nameinlink,noabbrev]{cleveref}   
\crefname{enumi}{}{}
\usepackage{tikz}
\usepackage{tikz-cd} 
\usetikzlibrary{shapes.geometric}
\usetikzlibrary{arrows}
\usepackage{graphicx}
\usepackage{fullpage}
\usepackage{enumerate}
\usepackage{latexsym}
\usepackage{amsfonts}
\usepackage[utf8]{inputenc}
\usepackage{amsmath}
\usepackage{amssymb}
\usepackage[all]{xy} 
\usepackage{mathrsfs} 
\usepackage{stmaryrd} 
\usepackage{amsopn}
\usepackage{eepic}
\usepackage{epic}
\usepackage{eucal}
\usepackage{epsfig}
\usepackage{psfrag}

\newtheorem{theorem}{Theorem}[section]
\newtheorem{lemma}[theorem]{Lemma}
\newtheorem{proposition}[theorem]{Proposition}

\newtheorem{corollary}[theorem]{Corollary}

\newtheorem{question}[theorem]{Question}

\theoremstyle{definition}
\newtheorem{definition}[theorem]{Definition}

\newtheorem{example}[theorem]{Example}
\newtheorem{remark}[theorem]{Remark}

\DeclareMathOperator{\Conv}{Conv}

\def\R{\mathbb{R}}

\def\C{\mathbb{C}}

\def\Z{\mathbb{Z}}

\renewcommand{\leq}{\leqslant}
\renewcommand{\geq}{\geqslant}

\DeclareMathOperator{\Vol}{Vol}

\DeclareMathOperator{\GL}{GL}
\DeclareMathOperator{\AGL}{AGL}

\newcommand{\one}{\mathbf{1}}

\newcommand{\zero}{\mathbf{0}}

\usepackage[disable,colorinlistoftodos, textwidth=2.5cm,textsize=small]{todonotes}

\begin{document}

\title{The structure of monotone blow-ups in symplectic toric geometry and a question of McDuff
}

\author{\'Alvaro Pelayo \,\,\,\,\,\,\,\, Francisco Santos}

\address{\'Alvaro Pelayo, Departamento de \'Algebra, Geometr\'ia y Topolog\'ia,
Facultad de Ciencias Matem\'aticas,
Universidad Complutense de Madrid, 28040 Madrid, Spain}
\email{alvpel01@ucm.es}

\address{Francisco Santos,
Departamento de Matem\'{a}ticas, Estad\'{i}stica y Computaci\'{o}n, Universidad de Cantabria, Av.~de Los Castros 48, 39005 Santander, Spain}
\email{francisco.santos@unican.es}



\begin{abstract}
Monotone polytopes, also known as smooth reflexive polytopes, are the polytopes associated to monotone symplectic toric manifolds and Gorenstein Fano toric varieties. We first show that the only monotone polytopes admitting blow-ups at vertices are the simplex and the result of a codimension-two blow-up in it (this is the polyhedral version of a result of Bonavero from 2002). Then we show that the $n$-simplex admits disjoint blow-ups at faces if and only if the faces are disjoint and have dimensions adding up to $n-1$ or $n-2$. These results answer a question posed by Dusa McDuff in 2011.
\end{abstract}

\maketitle

\setcounter{tocdepth}{3}

\section{Introduction}

Symplectic toric manifolds, also known as toric integrable systems,  have been intensively studied in the past four decades, from a combinatorial, symplectic and algebraic view point. Thanks to this, a fairly complete
dictionary exists which allows one to pose, and even solve, problems about them using a combination of techniques from these areas. For example, all the symplectic and algebraic information of the 
complex projective space of complex dimension $n$ can be detected in the standard simplex $n$-simplex $\Delta_n$, defined as the convex hull of 
the origin and the canonical basis vectors in $\mathbb{R}^n$. 

\subsection{A question about monotone symplectic blow-ups}
  
 A \emph{monotone polytope}, or a \emph{smooth reflexive polytope}, is a lattice polytope with the origin as its unique interior point and such that 
its polar is also a lattice polytope (equivalently, all facets are at lattice distance one from the origin). In this paper we 
answer a question of Dusa McDuff from 2011 about the structure of blow-ups of monotone polytopes, or equivalently, the structure of blow-ups
of monotone symplectic toric manifolds. 
Indeed, in her seminal paper on the topology of symplectic toric manifolds~\cite{McDuff}, McDuff asks:

\begin{question}[\protect{\cite[Question 2.7]{McDuff}}]
\label{q:2.7}
Is there a monotone polytope of dimension $d > 2$ for which one can make at least two monotone and disjoint blow ups of points, or, more generally, of any two faces of codimension $> 2$?
\end{question}

Via the Delzant correspondence monotone polytopes are the momentum polytopes of monotone symplectic toric manifolds, which makes them very significant in symplectic toric geometry. As McDuff says in here 
paper~\cite[page 151, paragraph 2]{McDuff}, \emph{``rather little seems to be known in general about their structure''}, and poses several questions about them, including Question~\ref{q:2.7}. 

\subsection{Polyhedral results}

We answer both parts of  Question~\ref{q:2.7} in a very strong way:
On the one hand, \emph{there are essentially only two monotone polytopes that admit monotone blow-ups at a point}. We attribute this result to Bonavero since, although our proof is different and completely elementary (as opposed to his more algebraic proof), our result translates to \cite[Theorem 1]{Bonavero}, via the dictionary relating lattice polytopes and toric varieties:

\begin{theorem}[\protect{\cite[Theorem 1]{Bonavero}}]
\label{thm:vertex-intro}
 The only monotone polytopes that admit \emph{one} monotone blow-up  at a point are $(n+1) \Delta_n$,  which is the only monotone $n$-simplex modulo $\AGL(n,\Z)$-equivalence,
 and the blow-up of a codimension-two face of $(n+1)\Delta_n$.
 \end{theorem}

\begin{corollary}
No monotone polytope of dimension $>2$ admits two monotone disjoint blow-ups at points.
\end{corollary}

On the other hand,  \emph{we classify completely the monotone blow-ups of arbitrary codimension in the monotone simplex $(n+1)\Delta_n$}:

\begin{theorem}
\label{thm:simplex-intro}
The monotone simplex $(n+1)\Delta_n$ admits disjoint blow-ups at two faces $F_1,F_2$ if and only if $F_1$ and $F_2$ are disjoint and their codimensions add up to either $n+2$ or $n+1$.
\end{theorem}

\begin{corollary}
For every $n\geq 4$, the monotone simplex $(n+1)\Delta_n$ admits two disjoint monotone blow-ups at faces of codimensions $>2$. 
\end{corollary}

\subsection{Symplectic and algebraic geometric results}

Theorem~\ref{thm:vertex-intro} and Theorem~\ref{thm:simplex-intro}  translate to the symplectic world as follows:

\begin{theorem}
\label{thm:vertex-symplectic}
The only monotone symplectic toric manifolds that admit a monotone toric blow-up at a point are the complex projective space $\C P^n$ and the result of a blow-up in $\C P^n$ at a $(\C^*)^n$-orbit of complex 
codimension two.
\end{theorem}

For a description of 
the symplectic geometry of $\C P^n$ and how it relates to the geometry/combinatorics of the $n$-simplex see for example~\cite[Section 7.1]{DuistermaatPelayo} and \cite[Example 5.2]{PelayoBAMS2017}.

\begin{theorem}
\label{thm:simplex-symplectic}
The complex projective space $\C P^n$ admits two disjoint monotone toric blow-ups if and only if the exceptional divisors of the blow-ups have 
complex dimensions adding up to $n-2$ or $n-1$.
\end{theorem}

In toric algebraic geometry, smooth reflexive polytopes correspond bijectively to (smooth) Fano toric varieties~\cite[Theorem 8.3.4]{CLS}. Theorem~\ref{thm:simplex-symplectic} applies then to Fano varieties, and Theorem~\ref{thm:vertex-symplectic} translates to the following statement:

\begin{theorem}[\protect{\cite[Theorem 1]{Bonavero}}]
The only Fano toric varieties that admit a monotone equivariant blow-up at a point are the complex projective space $\C P^n$ and the result of a codimension-two equivariant blow-up in $\C P^n$.
\end{theorem}

\subsection{Structure of the paper}
We provide concise and elementary proofs of these theorems in Section~\ref{theproofs} of the paper. In the previous sections we recall some basic concepts and results from symplectic geometry and polytope theory, to make it as self\--contained as
possible. In the final section we make some concluding remarks and pose an open question.

\subsection*{Acknowledgements}
The first author is funded by a BBVA (Bank Bilbao Vizcaya Argentaria) Foundation Grant for Scientific Research Projects with project title \emph{From Integrability to Randomness in Symplectic and Quantum Geometry}. He also thanks the Dean of the School of Mathematics of the Complutense University of Madrid Antonio Br\'u and the Chair of the Department of Algebra, Geometry and Topology Rutwig Campoamor for their support and  excellent resources he is being provided with to carry out the aforementioned project. 
The  second author is funded by grants PID2019-106188GB-I00  PID2022-137283NB-C21  of MCIN/AEI/10.13039/501100011033 and by project CLaPPo (21.SI03.64658) of Universidad de Cantabria and Banco Santander.

We thank M\'onica Blanco, Luis Crespo, Christian Haase, Benjamin Nill and Andreas Paffenholz for discussions on the topic of this paper.

We are grateful to the Department
of Mathematics, Statistics and Computation at the University of Cantabria for inviting the first author in July 2023 for a visit during which the majority of this paper was written. The first author is also thankful to Carlos Beltr\'an and Fernando Etayo for the hospitality during his visit and to the Governing Council of the UIMP (Universidad Internacional Men\'endez Pelayo), chaired by Rector Carlos Andradas, for the hospitality during his visit and  for providing him with a stimulating environment for the completion of this paper.

\section{Monotone symplectic toric manifolds and blow-ups}
In this section we recall what \emph{monotone symplectic manifolds} are and what are their corresponding \emph{monotone polytopes}.

For more details on symplectic geometry and Hamiltonian torus actions we refer to 
Cannas da Silva~\cite{AC},  Guillemin\--Sjamaar~\cite{GSj05}, McDuff--Salamon~\cite{McduffSalamon} and
Pelayo~\cite{PelayoBAMS2017}. 
See also 
Abraham\--Marsden~\cite{AbMa}, De Le\'on\--Rodrigues~\cite{LeRo} 
and Marsden\--Ratiu~\cite{MaRa} for references on symplectic geometry with an emphasis on the point of view of classical mechanics.

\subsection{Symplectic toric manifolds}

Symplectic toric manifolds are the symplectic analogues of toric varieties, and they have been intensively studied in equivariant symplectic geometry in the last four decades.

\begin{definition}[Symplectic toric manifold]
 A compact connected symplectic $2n$-dimensional manifold $(M,\omega)$ endowed with an effective Hamiltonian action of a $n$-dimensional torus $\mathbb{T}^n$ with momentum map 
 $$\mu\colon M \to \mathbb{R}^n$$ is called a \emph{symplectic toric manifold} and often denoted as a quadruple $(M,\omega,\mathbb{T}^n,\mu)$. 
\end{definition}

By a theorem (1982) of Atiyah \cite{At82} and Guillemin-Sternberg \cite{GS82} the image $\mu(M) \subset \mathbb{R}^n$ is a convex polytope, and obtained as the convex hull of the images of the fixed points of the $\mathbb{T}^n$-action on $M$ (in fact, their theorem was much more general, it applied to Hamiltonian $\mathbb{T}^m$-actions on compact connected $2n$-dimensional symplectic manifolds $(M,\omega)$ with $m \leq n$).

 Shortly after the work by the aforementioned authors, Delzant proved (1988) that any two such quadruples are isomorphic if and only if they have the same ``momentum polytope'' up to translations in $\mathbb{R}^n$ and ${\GL}(n,\mathbb{Z})$ transformations. Moreover, Delzant proved that $\mu(M)$  is a \emph{smooth polytope} (called a \emph{Delzant polytope} in symplectic geometry) and that this polytope completely classifies the symplectic toric manifold (up to ${\GL}(n,\mathbb{Z} \rtimes \mathbb{R}^n$ transformations). We refer to \cite{PELAYO2007}, \cite{Pelayo2023} or \cite{PelayoSantos2023} for a more detailed discussion of these results and a precise statement of the Delzant correspondence theorem.

\subsection{Monotone symplectic toric manifolds}

We are interested in the paper in symplectic toric manifolds which are also monotone in the following sense.

\begin{definition}[Monotone symplectic manifold]
A (compact connected) symplectic manifold is \emph{monotone} if there is $\lambda>0$ such that
\[
[\omega]=\lambda {\rm c}_1(M), 
\]
where ${\rm c}_1(M)$ denotes the first Chern class of $M$ (with respect to any almost complex structure compatible with $\omega$). 
\end{definition}

By rescaling, we may assume $\lambda=1$. Monotone symplectic manifolds have attracted significant interest in symplectic geometry and topology in recent years, and there are a number of important open problems concerning their structure. For recent works about monotone symplectic manifolds from the angle of toric geometry we refer to the papers by Fanoe~\cite{Fanoe}, Cho-Lee-Masuda-Park~\cite{CLMP22}, McDuff~\cite{McDuff} and Charton-Sabatini-Sepe~\cite{CSS23} as well as to the references in these papers.

McDuff  calls a Delzant polytope \emph{monotone} if it is the momentum map of a monotone toric manifold. 
She observes (\cite[Remark 3.2]{McDuffLagrangian}, \cite[p.~151, footnote]{McDuff}) that (modulo taking $\lambda=1$ plus a lattice translation) monotone polytopes are the same as the \emph{smooth reflexive polytopes} that have been studied in algebraic geometry in the context of toric varieties~\cite{Batyrev, GodinhoHeymannSabatini, HaaseMelnikov06, Nill}, since they correspond bijectively to Fano toric varieties~\cite[Theorem 8.3.4]{CLS}.
Here a \emph{smooth} polytope is a simple lattice polytope $P$ whose normal fan is unimodular and a smooth polytope is called \emph{reflexive} if it contains the origin in its interior and the dual polytope of $P$ is also a lattice polytope.

In this setting, smooth reflexive polytopes are the moment polytopes of Gorenstein Fano varieties~\cite{Nill}, \cite[Theorem 8.3.4]{CLS}, and their duality properties are related to  mirror symmetry~\cite{Batyrev, Witten}.
The relation bewteen the algebraic and symplectic viewpoints on toric manifolds is explained in detail, e.g., \cite{De1988,DuistermaatPelayo}.
See~\cite{CSS23} for the relation between monotone and Fano in the non-toric case.

\subsection{Monotone blow-ups of Delzant polytopes and symplectic toric manifols}

Here we briefly recall the blowing-up construction in symplectic (toric) geometry, which is a special case of Lerman's symplectic cutting \cite{LERMAN}. We follow the outline in McDuff \cite[Sections 2.2 and 3.2]{McDuff}.

 Let $(M,\omega,\mathbb{T}^n, \mu \colon M \to \mathbb{R}^n)$ be a symplectic toric manifold with momentum polytope 
\[
P:=\mu(M)=\{ x \in \mathbb{R}^n \,|\,  u_i \cdot x   \leq b_i \,\, \forall i \in \{1,  \ldots,N\} \}.
\]

At the level of polytopes, for each face $F$ of $P$ and positive number $\epsilon$, the  \emph{(toric) blow-up of $P$ along $F$ of size $\epsilon$} is
the polytope $P_{F,\epsilon}$ obtained by adding to the inequality description of $P$ a new facet 
with inequality $u_0 \cdot x \le b -\epsilon$, where $u_0$ is the sum of the normal vectors of facets of $P$ containing $F$ and $b$ is the supporting value of $u_0$ in $P$ (that is, $b=\max\{u_0\cdot x: x \in P\}$).
 The parameter $\epsilon >0$ needs to be small enough so that all vertices of $P$ not in $F$ are still in $P_{F,\epsilon}$.
 We explain this in more detail in  Definition~\ref{pl}.

The symplectic manifold corresponding to the blown-up polytope $P_{F,\epsilon}$ is constructed from $(M,\omega)$ by ``excising'' from it the set
$$
\mu^{-1}(\{x \in P\,|\,   u_0 \cdot x  > b-{\epsilon}\}),
$$
and collapsing the boundary along its characteristic flow. This resulting manifold inherits a toric structure and corresponding momentum map, which are completely 
\textcolor{red}{determined by} the polytope $P_{F,\epsilon}$ via Delzant's correspondence. 

We refer to Lerman \cite{LERMAN} and McDuff-Tolman \cite[Section 2.4, in particular Remark 2.4.4]{MCDUFFTOLMAN} for a more detailed discussion of these constructions and Pelayo \cite[Section 3]{PELAYO2007} for a discussion on the possible sizes of the $\epsilon$-blow-ups.

By a \emph{monotone blow-up} we mean a blow-up in a monotone manifold that results in another monotone manifold.

\section{Monotone, also known as smooth reflexive, polytopes}

In this section we review the notions of Delzant polytope and monotone polytope and derive some basic properties of them. For the basic theory of polytopes we refer to the textbooks \cite{triangbook, Ziegler}. For lattice polytopes see
\cite{HaaseNillPaffenholz}. For their relation to toric varieties see \cite{CLS,Ewald, Fulton, Oda}.

\subsection{Delzant polytopes}

\begin{definition}[Delzant polytope]
A \emph{Delzant polytope} or \emph{smooth polytope}\footnote{In some literature a smooth polytope is required to have integer vertices so that ``smooth''=``lattice Delzant polytope. Since we are interested only in the reflexive case, which necessarily implies integer vertices, this distinction is not important for us} 
is a simple full-dimensional polytope in $\R^n$ with rational edge directions and such that the primitive edge-direction vectors at each vertex form a basis of the lattice $\Z^n$. Equivalently, it is a polytope with a simplicial and unimodular normal fan.
\end{definition}

For each facet $F$ of a rational polytope $P$ let us denote $u_F$ the primitive (exterior) normal vector to $F$,
so that the irredundant inequality description of $P$ is
\[
P=\{ x \in \mathbb{R}^n \,|\,  u_F \cdot x   \leq b_F, \,\, \text{ $F$ a facet of $P$} \},
\]
for some real constants $b_F$.
We extend the notation to lower dimensional faces. For a face $G$ of $P$ we denote
\[
u_G:= \sum_{G\subset F,\ F \text{ a facet}} u_F.
\]
We call $u_G$ the \emph{central normal vector} of the face $G$.

We always consider polytopes modulo $\AGL(n,\Z)$-equivalence. That is, $P$ si considered equivalent to $P'$ if (and only if) there is an integer $n\times n$ matrix $A$ of determinant $\pm 1$ (that is, with integer inverse) and an integer translation vector $t\in \Z^n$ such that the map $x\mapsto Ax+t$ sends $P$ to $P'$.

The \emph{length} of a segment with rational direction will always be measured with respect to the lattice. That is, the length of a segment $ab$ is the unique constant $\ell>0$ such that the vector $\frac{1}{\ell}(b-a)$ is primitive. 

Observe that modulo $\AGL(n,\Z)$ if $P$ is a smooth polytope and $v$ is a vertex of it, without loss of generality we can assume that $v=(0,\dots,0)$ and that the facets meeting at $v$ are defined by $x_i\geq 0$ for every coordinate.

\begin{example}[The smooth simplex]
\label{exm:simplex}
If $P$ is a smooth simplex and, as said above, we assume without loss of generality that the origin is a vertex and that $n$ of its $n+1$ facets have normals $-{\rm e}_1,\dots,-{\rm e}_n$ then the remaining facet must have normal vector $u=(u_1,\dots,u_n)$ with all $u_i$ strictly positive (in order to get a bounded polytope) and equal to $1$ (in order for $n$ to be unimodular with any $n-1$ of the other $n$ vectors. Hence, $P$ is of the form
\[
\left \{x\in \R^n: x_i\geq 0 \,\, \forall i \right\}\cap\Big\{ x\in \R^n: \sum_{i=1}^n x_i \leq b\Big\}
 =\Conv\{\zero, b {\rm e}_1,\dots , b {\rm e}_n\}.
\]
for some positive constant $b$ which equals the length of every edge in $P$. We call this polytope  the smooth simplex of size $b$. If $b=1$ we call it the smooth unimodular simplex and denote it $\Delta_n$.
\end{example}

\begin{definition}[Blow-up of a Delzant polytope] \label{pl}
Let $F$ be a face of a Delzant polytope $P$. Let $b\in \R$ be the unique constant such that $a_F\cdot x = b$ for every $x \in F$, and let $\epsilon >0$ be a constant such that $u_F \cdot v < b-\epsilon$ for every vertex of $P$ not in $F$. The \emph{blow-up} of $P$ at $F$ of size $\epsilon$ is the polytope 
\[
P_{F,\epsilon} := P \cap\{x\in \R^n: u_F\cdot x \leq b-\epsilon\}.
\]
\end{definition}

It is easy  to check that $P_{F,\epsilon}$ is still a Delzant polytope. Its combinatorics is that the face $F$ is replaced by a new facet combinatorially isomorphic to $F\times \Delta_{k-1}$, where $k$ is the codimension of $F$ and $\Delta_{k-1}$ denotes the $(k-1)$-simplex, and every face $F'$ of $P$ containing $F$ suffers the same change. Here and in what follows two polytopes are said \emph{combinatorially isomorphic} if their posets of faces are isomorphic.

Let $L$ be the linear subspace of codimension $k$ parallel to $F$.
Observe that the linear projection $\pi_F: \R^n \to \R^n/L$ in the direction of $F$ sends $P$ to a polytope that may perhaps not be globally smooth but which is smooth at its vertex $v:=F/ L$. If $P'$ is a blow-up of $P$ at $F$ then $\pi(P')$ is obtained from $\pi(P)$ by cutting out a smooth $k$-simplex $\Delta$ with vertex $v$ and of size $\epsilon$. The fiber $\pi^{-1}(x) \cap P$ of every $x\in \Delta$ is a polytope combinatorially isomorphic to (and with the same normal fan as) $F$.\footnote{In other words, $P\setminus P'$ is an example of what McDuff calls a bundle over $\epsilon \Delta_k$ with fiber $F$.}

The definition easily implies the following lemma:
\begin{lemma}
\label{lemma:edges}
A blow-up at $F$ of size $\epsilon$ can be performed if, and only if, every edge with exactly one end in $F$ has length strictly greater than $\epsilon$.
\end{lemma}

\subsection{Smooth reflexive polytopes}

\begin{definition}[Reflexive polytope]
A \emph{reflexive polytope} is a lattice polytope (i.e., a polytope with integer vertices) and whose dual also has integer vertices. Equivalently, a lattice polytope is reflexive if and only if every facet-supporting hyperplane is of the form $u_F \cdot x=1$, where $u_F$ is the primitive vector normal to the facet.
\end{definition}

Observe that every reflexive polytope has the origin as the unique interior lattice point. In particular,  for reflexive polytopes $\AGL(n,\Z)$-equivalence is the same as $\GL(n,\Z)$-equivalence. For example, it follows from the description in Example~\ref{exm:simplex} that the only monotone simplex (modulo $\GL(n,\Z)$-equivalence) is 
\[
\{x\in \R^n : x_i\geq -1 \ \forall i \text{ and } \sum_i x_i \leq 1\} = -\one + 
(n+1) \Delta_n \cong 
(n+1) \Delta_n,
\]
were $\one := (1,\dots,1) \in \R^n$, and we use ``$\cong$'' to express combinatorial isomorphism.
We call this \emph{the} monotone simplex. Observe that all its edges have length $n+1$.

\begin{definition}[Monotone polytope]
A \emph{monotone polytope} is a polytope that is both smooth and reflexive. A \emph{monotone blow-up} in a monotone polytope is a blow-up that results in another monotone polytope.
\end{definition}

Monotone polytopes receive this name in the toric symplectic literature because they coincide (modulo a global dilation) with the moment polytopes of monotone  symplectic toric manifolds. (See, e.g., \cite{CSS23,McDuff}).
In the lattice polytope literature, and in the algebraic geometry literature, they are usually called \emph{smooth reflexive}.

It is easy to prove that there are finitely many reflexive polytopes in every fixed dimension~\cite{LagariasZiegler}.
(Here and elsewhere rational polytopes are always considered modulo $\AGL(n,\Z)$-equivalence). In particular, there are also finitely many monotone polytopes. They have been enumerated up to dimension 9 and their number equals
\[
\begin{array}{c|ccccccccc}
\text{dimension}& 1& 2& 3& 4& 5& 6& 7& 8& 9 \\
\hline
\text{\# of monotone polytopes} &1& 5& 18& 124& 866& 7622& 72256& 749892& 8229721\\
\end{array}
\]
The enumeration up to dimension eight was done by
{\O}bro~\cite{OebroSmoothFano}.
It was extended up to dimension nine by Lorenz and Paffenholz~\cite{LorenzPaffenholz}.

By definition, in a reflexive polytope every facet-supporting hyperplane is of the form $\{u_F\cdot x=1\}$. For smooth ones a similar property extends to lower-dimensional faces:

\begin{lemma}
Let $F$ be a face of codimension $k$ in a monotone polytope $P$ and let $u_F$ be its central normal vector. Then $F$ is contained in the hyperplane $\{u_F\cdot x = k\}$.
\end{lemma}

\begin{proof}
If $F_1,\dots,F_k$ are the facets containing $F$, we have that $u_{F_i} x=1$ holds in $F_i$ (hence in $F$) for every $i$. Since $u_F=\sum_iu_{F_i}$ we have that $u_{F} x=k$ holds in $F$.
\end{proof}

The lemma implies the following property:

\begin{lemma}
\label{lemma:blowups}
Let $F$ be a face of codimension $k$ in a montotone polytope $P$. The following properties are equivalent:
\begin{enumerate}
\item[\rm (1)] There is a monotone blow-up at $F$.
\item[\rm (2)]  Every edge with exactly one end in $F$ has length at least $k$.
\end{enumerate}
Moreover, if this happens then the monotone blow-up at $F$ has size $k-1$.
\end{lemma}

Observe that facets can be blown-up but the blow-up simply ``pushes'' the facet toward the interior of the polytope, maintaining its normal vector and not changing the combinatorial type; moreover,  a ``monotone blow-up of codimension 1'' has size 0, by the lemma above, which means it does nothing at all to the polytope. 
Hence we always assume that the codimension of the face to be blown-up is at least two.

\begin{proof}
The ``moreover'' follows from the previous lemma and the fact that n order for the blow-up to be monotone we need the facet created by it to be contained in the hyperplane $\{u_F \cdot x=1\}$. The equivalence of (1) and (3) follows from the fact that this blow-up can be performed if and only if $$u_F\cdot  x < k-(k-1) = 1$$ holds for every vertex not in $F$, but ``$<1$'' is equivalent to ``$\leq 0$'' because vertices have integer coordinates. The same argument, now using Lemma~\ref{lemma:edges}, shows the equivalence of (1) and (2).
\end{proof}

\begin{remark}
The reason why Dusa McDuff asks for the dimension of $P$ and the codimension of the blow-ups to be at least three in Question~\ref{q:2.7} is that in the monotone triangle the three vertices can simultaneously be blown-up, and in the monotone tetrahedron two opposite edges or an edge and a vertex not incident to one another can be blown-up too. McDuff (with the help of A.~Paffenholz, see~\cite[p. 160, footnote]{McDuff}) had checked that in dimension three the answer to the question is negative: no monotone $3$-polytope admits two disjoint blow-ups of codimension three, that is, at vertices. 

We have redone the computation of all possible blow-ups among the eighteen monotone polytopes and obtained the diagram in Table~\ref{table:3d}.
In the table, the polytope labelled {\Large\textcircled{\small N}} is the one that can be obtained with the Sage command  {\tt ReflexivePolytope(3,N-1)}. Figure \ref{fig:maximal3d} shows the five maximal ones.\footnote{We here mean ``maximal'' with respect to containment or, equivalently, ``minimal'' with respect to their sets of normal vectors. In dimension three monotone polytopes with these properties coincide with the ones that cannot be obtained as a blow-up of another monotone polytope, but in higher dimension this may not be true.}
\end{remark}

\newcommand{\polylabel}[1]{ \stackrel{   {\Large\textcircled{ {\hspace{-11pt} \small #1}}}    }}
\newcommand{\poly}[2]{ \stackrel{   {\Large\textcircled{ {\hspace{-11pt} \small #1}}}    }{\text{\tiny ($\Vol$=#2)} }}
\renewcommand{\tabcolsep}{0.cm}
\begin{figure}[htb]
\begin{tabular}{ccccc}
\includegraphics[scale=0.28]{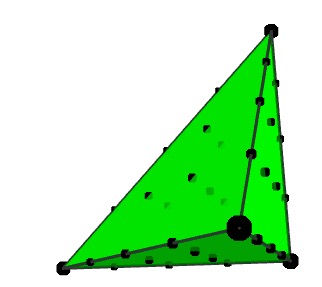}&
\includegraphics[scale=0.17]{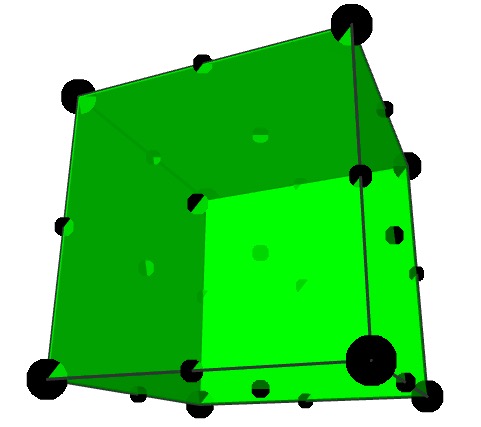}&
\includegraphics[scale=0.17]{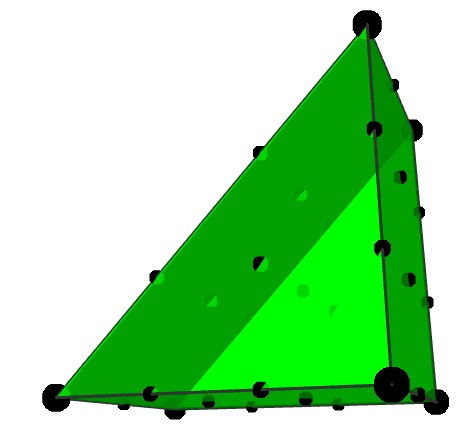}&
\includegraphics[scale=0.23]{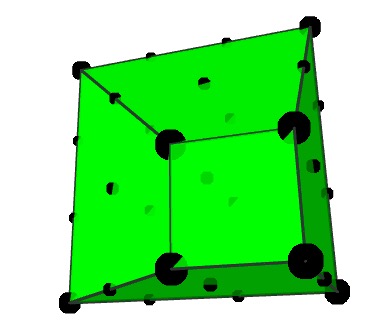}&
\includegraphics[scale=0.21]{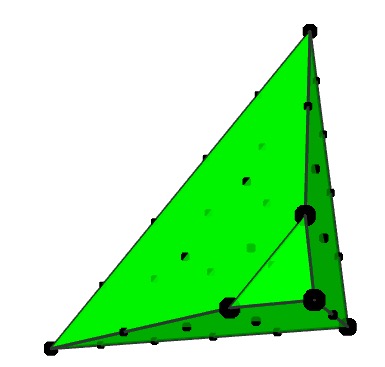}\\
$\polylabel{1}{}$ & $\polylabel{31}{}$ & $\polylabel{5}{}$ & $\polylabel{28}{}$ & $\polylabel{8}{}$ \\
\end{tabular}

\caption{The five maximal monotone $3$-polytopes: simplex, cube, triangular prism, slanted cube, and slanted prism. Pictures created with \href{https://sagecell.sagemath.org}{SageMathCell}.}
\label{fig:maximal3d}
\end{figure}

\begin{table}[htb]
\begin{tikzcd}[column sep=small]
4\ \text{facets}	& & \poly{1}{64}  \arrow[d,"v"] \arrow[dr,""] & & & &  \\
5\ \text{facets}	& \poly{8}{62}  \arrow[d,""]  & \poly{6}{56} \arrow[dl,""] \arrow[d,""] \arrow[dr,""] &\poly{7}{54} \arrow[dl,""] \arrow[d,"v"] \arrow[dr,""] & \poly{5}{54} \arrow[dl,""] \arrow[dr,""] & & & \\
6\ \text{facets}	& \poly{27}{50}  & \poly{30}{50} \arrow[dr,""] \arrow[drr,""] & \poly{26}{46} \arrow[dr,""] & \poly{25}{44} \arrow[d,""] & \poly{30}{48} \arrow[dl,""] \arrow[d,""] \arrow[dr,""]& \poly{28}{52} \arrow[dl,""]& \poly{31}{48} \arrow[dl,""]\\
7\ \text{facets}	& & & \poly{84}{46} \arrow[dr,""]& \poly{82}{40} \arrow[d,""]& \poly{83}{44} \arrow[dl,""]& \poly{85}{42} \arrow[d,""]& \\
8\ \text{facets}	& & & & \poly{\text{\tiny 220}}{ 36} && \poly{\text{\tiny 219}}{\text{\small 36}} & \\
\end{tikzcd}
\caption{\small Diagram of the blow-ups among  monotone $3$-polytopes.
 Each node represents one of the eighteen monotone $3$-polytopes, with its {\tt Sage} label and its volume (volume is normalized to the unimodular simplex, in order for it to be always an integer). 
Arrows represent monotone blow-ups. The two arrows marked $v$ are the only two blow-ups at vertices, as predicted by Theorem~\ref{thm:vertex-intro}.
The polytopes are arranged in rows according to their number $f_2$ of 2-dimensional faces. 
Each blow-up, by definition, adds one to the number of facets. 
In any simple $3$-polytope the whole $f$-vector can be obtained from $f_2$ as $f_0=2f_2-4$ and $f_1=3f_2-6$.}
\label{table:3d}
\end{table}

\section{Proof of the main results} \label{theproofs}

Blowing up monotone manifolds in a monotone way, imposes great constraints. 
For example, in view of Lemma~\ref{lemma:blowups} we can perform a monotone blow up at the vertex $v$ (here $k=n$) of a monotone polytope if and only if every edge of the polytope meeting $v$ has length at least equal to $n$.
Similarly, for two disjoint blow-ups to be possible at faces $F_1$ and $F_2$ that are connected by an edge and of codimensions $k_1$ and $k_2$ one needs the length of the edge to be at least 
$$k_1+k_2 \ge n+2.$$ These two fundamental facts are exploited in this section to prove, respectively, Theorems~\ref{thm:simplex-intro} and~\ref{thm:vertex-intro}.

\subsection{Blow-ups in the monotone simplex.}

Since the monotone simplex has a complete graph, every two disjoint faces $F_1$ and $F_2$ are joined by an edge.
By Lemma~\ref{lemma:blowups}, if the faces' codimensions are $k_1$ and $k_2$ respectively, the blow-ups will cut that edge at distances $k_1-1$ and $k_2-1$. Thus, for the blow-ups to be disjoint one needs $k_1 + k_2-1$ to be strictly smaller than the length of the edges in the simplex, which is $n+1$. Hence:

\begin{proposition}
\label{prop:simplex}
Let $F_1,\dots,F_m$ are disjoint faces of the monotone simplex $(n+1)\Delta_n$, of codimensions $k_1,\dots,k_m$. Then, the following properties are equivalent:
\begin{enumerate}
\item[\rm (1)] The monotone blow-ups at $F_1,\dots, F_m$ are disjoint. 
\item[\rm (2)]  The monotone blow-ups at $F_1,\dots, F_m$ are pairwise disjoint. 
\item[\rm (3)]  $k_i+k_j \leq n+2$ for every $i,j\in\{1,\dots,m\}$.
\end{enumerate}
\end{proposition}

This directly implies Theorem~\ref{thm:simplex-intro}:

\begin{proof}[Proof of Theorem~\ref{thm:simplex-intro}]
Disjoint faces $F_1,F_2$ of dimensions $d_1,d_2$ exist if and only if 
$$d_1 +d_2 + 2 \leq n+1,$$ since the face $F_i$ will contain $d_i+1$ vertices. Thus, since $k_i=n-d_i$, for two disjoint monotone blow-ups of codimensions $k_1$ and $k_2$ to be possible in the monotone simplex we need $k_1 +k_2  \geq n+1$, in addition to the inequality of  Proposition~\ref{prop:simplex}.
\end{proof}

\begin{corollary}
The only monotone $n$-simplex that admits \emph{three} disjoint blow-ups is the triangle, in which the three vertices can be blown-up simultaneously. 
\end{corollary}

\begin{proof}
Suppose that three blow-ups are possible. Then we have $k_1+k_2 \leq n+2$ or, equivalently, $d_1 + d_2 \geq n+2$, which implies that the first two faces $F_1$ and $F_2$ already use all except perhaps one of the $n+1$ vertices in the $n$-simplex. The only possibility then is that $F_3$ equals that vertex, so that $k_3=n$. But then the inequalities  $k_1+k_3\leq 2$ and $k_2+k_3\leq 2$ imply $k_1,k_2\leq 2$, and $k_1+k_2 \leq n+2$ implies $n\leq 2$.
\end{proof}

\subsection{Blow-ups at vertices.}

We first need the following result about lengths of edges in monotone polytopes.

\begin{lemma}
\label{lemma:lengths}
Let $v$ be a vertex in a monotone $n$-polytope $P$. If the blow-up at $v$ is possible (that is, if all edges incident to $v$ have length at least $n$) then all edges incident to $v$ have length $n$ or $n+1$.
\end{lemma}

\begin{proof}
Without loss of generality assume that $v=-\one$, and that the edges incident to it go in the positive coordinate directions. Let $v_1,\dots,v_n$ be the neighbors of $v$ and let $\ell_i$, with $i=1,\dots,n$, denote the length of the edge $vv_i$. We have that
\[
v_i = -\one + \ell_i {\rm e}_i.
\]

For each vertex $v_i$ let $F_i$ be the unique facet containing $v_i$ but not $v$, and let
$u_i$ be the primitive normal vector to $F_i$.
We concentrate on $v_1$ and $u_1$ for the rest of the proof, but the arguments obviously apply to every $i$. 

Let us denote $u_i=(\lambda_1,\dots,\lambda_n)$ the coordinates of $u_1$.
Since the other facet-normals of facets containing  $v_1$ are $-{\rm e}_2,\dots,-{\rm e}_n$,  smoothness implies that $\lambda_1=1$. Reflexiveness, in turn,  implies that
\[
1 = u_1\cdot v_1 = u_1 \cdot (-\one + \ell_1 {\rm e}_1) = -1 -\sum_{i=2}^n \lambda_i + \ell_1,
\]
so that $$\sum_{i=2}^n\lambda_i = \ell_1-2.$$

For each $i\in \{2,\dots,n\}$ with $\lambda_i >0$ consider the edge-vector starting from $u_1$ and in the $2$-face spanned by $v$, $v_1$ and $v_i$. This vector is orthogonal to $u_1$, 
so it is parallel to $-\lambda_i {\rm e}_1 + {\rm e}_i$. 
In particular, the $2$-face in question is contained in the triangle with vertices $v$, $v_1$ and $-\one + \ell_1/\lambda_i$, which implies $\ell_i \leq \ell_1/\lambda_i$. Hence 
\begin{align*}
\label{eq:inverse}
{\lambda_i} \leq  \frac{\ell_1}{\ell_i}.
\end{align*}
This equation is trivially satisfied also if $\lambda_i \leq 0$, so we get that
\[
 {\ell_1-2} = \sum_{i=2}^n \lambda_i \leq \sum_{i=2}^n\frac{\ell_1}{\ell_i} \leq \ell_1\frac{n-1}{n},
\]
where the last inequality comes from the fact that $\ell_i\geq n$ for every $i$.

That is:
\[
\frac{\ell_1-2}{\ell_1} \leq \frac{n-1}n,
\]
which implies $\ell_1\leq 2n$. 
We claim that $\lambda_i \leq 1$ for every $i$, which will finish the proof since then 
\[
\ell_1 = \sum_{i=2}^n\lambda_i +2 \leq (n-1)+2=n+1.
\]

To prove the claim, suppose in order to seek a contradiction that $\lambda_i\geq 2$ for some $i$. Together with $\ell_1\leq 2n$ this gives
\[
n \leq \ell_i \leq \frac{\ell_1}{\lambda_i} \leq \frac{2n}{2} =n,
\]
which implies that the two inequalities are equalities; $\lambda_i=2$, $\ell_i=n$, and $\ell_1=2n$. But this says that  the $2$-face containing $v$, $v_1$ and $v_i$ must 
 \emph{equal} the triangle discussed above, with two edges in the coordinate directions $1$ and $i$ and third edge parallel to $-2 {\rm e}_1 + {\rm e}_i$. 
This is impossible since that triangle is not smooth: the vectors ${\rm e}_i$ and $-2{\rm e}_1 + {\rm e}_i$ cannot be part of a unimodular basis.

This finishes the proof, but let us make the following additional remark to prepare for the proof of our next result:
The inequality $\lambda_i \leq 1$ for all $i$ together with the fact that $$\sum_{i=2}^n \lambda_i =\ell_1 -2 \in \{n-2,n-1\}$$ gives the following two possibilities for $u_1$: either all $\lambda_i$'s equal 1, or all $\lambda_i$'s equal 1 except for one of them which equals zero.
That is:
\begin{quote}
\it
If all edges from a vertex $v$ of a monotone $n$-polytope have length at least $n$, then the  normals to the facets $F_1,\dots, F_n$ containing a neighbor of $v$ but not $v$ are
 (in the coordinate system where edges from $v$ go in the positive coordinate directions) all of the form $\one$ or $\one-{\rm e}_i$.
\qedhere
\end{quote}
\end{proof}

We are now ready to prove Theorem~\ref{thm:vertex-intro}:

\begin{proof}[Proof of Theorem~\ref{thm:vertex-intro}]
Let us keep the notation of the previous proof, so that the vertex that admits a blow-up is $v:=-\one$, the neighbors of $v$ are
\[
v_i := -\one + \ell_i {\rm e}_i
\]
with $\ell_i\in \{n,n+1\}$ for all $i$, and $F_i$ and $u_i$ denote the facet containing $v_i$ but not $v$ and its normal vector, respectively.

As said at the end of the previous proof, for each $i$ we have that either $u_i=\one$ or $u_i =\one -{\rm e}_j$ for some $j\ne i$. Moreover these two possibilities correspond exactly to $\ell_i=n+1$ and $\ell_i=n$, respectively. (This last statement follows from the previous proof, but it also follows immediately from $u_i \cdot v_i =1$).

If all lengths $\ell_i$ are equal to $n+1$, that is, if all facets $F_i$ have the same normal vector $\one$, then all the $v_i$ lie in one and the same facet, contained in the hyperplane $\one \cdot x = 1$. $P$ equals the monotone simplex.

So, for the rest of the proof we assume that (at least) one of the lengths equals $n$, so there is a facet $F$ defined by $$(\one-{\rm e}_j)\cdot x \le1$$ for some $j$. This implies that \emph{every} $\ell_i$ with $i\ne j$ equals $n$ (because  $(\one-{\rm e}_j)\cdot(-\one + \ell_i {\rm e}_i) \leq 1$ gives $\ell_i \leq n$) and the corresponding $v_i$'s lie in $F$. That is, $F$ contains all but one of the neighbors of $v$.
Now,  at the remaining neighbor $v_j$ the normal to its facet $F_j$ can not be of the form $\one-{\rm e}_k$, because such a facet should again contain all the neighbors of $v$ except $v_k$, implying that every neighbor of $v$ except for $v_j$ and $v_k$ lie in (at least) $n+1$ facets: the $n-1$ that it has in common with $v$ plus the two facets $F$ and $F_j$. (Observe that here is where we use $n\geq 3$, since we are using that $v$ has at least one neighbor other than $v_j$ and $v_k$.)

Thus, $P$ is contained in the polytope $Q$ defined by the following $n+2$ inequalities: $-x_i\geq 1$ for all $i$, $\sum_i x_i \leq 1$ and  $\sum_{i\ne j} x_i \leq 1$ for some $j$. 
This polytope $Q$ is the blow-up of the monotone simplex at the codimension-two facet opposite to $v$ and $v_j$, and our claim is that $P$ actually equals $Q$.
To prove the claim observe that $Q$ has the following $2n$ vertices (it is combinatorially a prism over an $(n-1)$-simplex): $v$, $v_j=v+(n+1) {\rm e}_j$, and, for each $i\ne j$, 
the two vertices $v_i=v+ n {\rm e}_i$ and $v_i +{\rm e}_j =v+n{\rm e}_i +{\rm e}_j$. Our only remaining task is to show that the vertices of this last form are also in $P$; this will imply $P=Q$ since $P$ will be a polytope contained in $Q$ and containing all the vertices of $Q$.

Recall that the facets at $v_i$ have the following $n$ normal vectors: the vector ${\rm e}_k$ for every $k\ne i$ and the vector $\one - {\rm e}_j$ normal to $F$. Forgetting the facet with normal ${\rm e}_j$ we get that the intersection of the remaining $n-1$ is an edge parallel to the $j$-coordinate. Thus: $P$ has an edge going from $v_i$ in the positive $j$-coordinate direction. This edge necessarily contains the next lattice point, which is precisely the vertex $v_i +{\rm e}_j$ of $Q$ we were searching for.
\end{proof}

\section{Concluding remarks and a question}

In this paper have have exploited the deep connection between symplectic geometry, algebraic geometry and combinatorics, via the dictionaries which allow the translation of certain problems
from one area to another, as it has been discovered in recent years. 

We have been concerned with smooth reflexive polytopes, which are geometric combinatorial objects that have
analogues both in algebraic and symplectic geometry, and which appear also in the context of mirror symmetry. The theorems we have shown essentially classify the
structure of monotone blow-ups of such polytopes, and,  via de dictionary between polytopes and manifolds/varieties we have discussed,  the structure of monotone blow-ups
of monotone symplectic toric manifolds and of equivariant blow-ups of toric Fano varieties.

In recent years many groups have used these connections to advance their work, as this
fruitful interaction has allowed to translate certain problems into a language and setting in which they can be more tractable (or for which there are better developed techniques). One important example of
these interactions appears emphasized in McDuff's seminal paper~\cite{McDuff} on the topology of symplectic toric manifolds,  which was the inspiration for the present paper, and 
where she uses both symplectic geometric arguments and polytope/combinatorial arguments 
to prove a number of results about monotone symplectic toric manifolds.

We pose the following question, motivated by the results of the present paper:

\begin{question}
Does there exist a monotone $n$-polytope with two monotone blow-ups of codimensions adding to more than $n+2$?
\end{question}

We believe the answer to be negative, which would give a common generalization of Theorems~\ref{thm:vertex-intro} and~\ref{thm:simplex-intro}.

 In fact, a positive answer to the question might imply that a complete classification of monotone polytopes that admit two disjoint monotone blow-ups is doable. The list should include simplices and also products, wedges, and bundles over them, with certain restrictions on the dimensions of the factors.

\bibliographystyle{amsplain}

\end{document}